\pdfoutput=1
\documentclass[reqno,oneside,letterpaper]{amsart}
\usepackage{mathrsfs}
\usepackage{amssymb}
\usepackage{wasysym}

\usepackage[utf8]{inputenc}
\usepackage{amsmath, amssymb,bm, cases, mathtools, thmtools}
\usepackage{verbatim}
\usepackage{graphicx}\graphicspath{{figures/}}
\usepackage{multicol}
\usepackage{tabularx}
\usepackage[usenames,dvipsnames]{xcolor}
\usepackage{mathrsfs}
\usepackage{url}
\usepackage[normalem]{ulem}
\usepackage{xstring}
\usepackage{enumitem}

\usepackage[%
    minnames=1,maxnames=99,maxcitenames=3,
    style=alphabetic,
    doi=false,url=false,
    firstinits=true,hyperref,natbib,backend=bibtex]{biblatex}
\renewbibmacro{in:}{%
  \ifentrytype{article}{}{\printtext{\bibstring{in}\intitlepunct}}}
\bibliography{biblio_sec}

\usepackage[colorlinks,citecolor=blue,urlcolor=blue,linkcolor=RawSienna]{hyperref}
\usepackage{hypernat}
\usepackage{datetime}

\DeclareMathAlphabet\EuRoman{U}{eur}{m}{n}
\SetMathAlphabet\EuRoman{bold}{U}{eur}{b}{n}

\usepackage{euscript,microtype}

\usepackage[capitalize]{cleveref}

\crefname{assumption}{Assumption}{Assumptions}
\crefname{claim}{Claim}{Claims}

\makeatletter
\let\reftagform@=\tagform@
\def\tagform@#1{\maketag@@@{\ignorespaces\textcolor{gray}{(#1)}\unskip\@@italiccorr}}
\renewcommand{\eqref}[1]{\textup{\reftagform@{\ref{#1}}}}
\makeatother

\definecolor{WowColor}{rgb}{.75,0,.75}
\definecolor{SubtleColor}{rgb}{0,0,.50}

\newcounter{margincounter}

\declaretheorem[style=plain,numberwithin=section,name=Theorem]{theorem}
\declaretheorem[style=plain,sibling=theorem,name=Lemma]{lemma}

\declaretheorem[style=plain,sibling=theorem,name=Claim]{claim}

\declaretheorem[style=definition,sibling=theorem,name=Definition]{definition}

\declaretheorem[style=definition,sibling=theorem,name=Example]{example}

\crefformat{conditionINNER}{#2(#1)#3}
\crefmultiformat{conditionINNER}
  {(#2#1#3)}
  { and~(#2#1#3)}
  {, (#2#1#3)}
  { and~(#2#1#3)}
\Crefformat{conditionINNER}{Condition~#2(#1)#3}
\Crefmultiformat{conditionINNER}
  {Conditions~(#2#1#3)}
  { and~(#2#1#3)}
  {, (#2#1#3)}
  { and~(#2#1#3)}

\declaretheoremstyle[
    spaceabove=-6pt,
    spacebelow=6pt,
    headfont=\normalfont\bfseries,
    bodyfont = \normalfont,
    postheadspace=1em,
    qed=$\square$,
    headpunct={{}}]{myproofstyle}

\numberwithin{equation}{section}
\numberwithin{theorem}{section}

\usepackage{amssymb}
\usepackage{mathrsfs}
\usepackage{leftidx}

\usepackage{setspace}

\def\[#1\]{\begin{align}#1\end{align}}
\def\*[#1\]{\begin{align*}#1\end{align*}}

\newcommand{\Nats}{\mathbb{N}}

\DeclareMathOperator*{\newlim}{\mathrm{lim}\vphantom{\mathrm{infsup}}}

\DeclareMathOperator*{\newmax}{\mathrm{max}\vphantom{\mathrm{infsup}}}
\DeclareMathOperator*{\newinf}{\mathrm{inf}\vphantom{\mathrm{infsup}}}
\DeclareMathOperator*{\newsup}{\mathrm{sup}\vphantom{\mathrm{infsup}}}
\renewcommand{\lim}{\newlim}

\renewcommand{\max}{\newmax}
\renewcommand{\inf}{\newinf}
\renewcommand{\sup}{\newsup}

\newcommand{\cF}{\mathcal F}
\newcommand{\cG}{\mathcal G}

\newcommand{\BorelSets}[1]{\mathcal{B}[#1]}
\newcommand{\NSE}[1]{{^{*}#1}}
\newcommand{\ST}{\mathsf{st}}

\newcommand{\PowerSet}{\mathscr{P}}

\newcommand{\cA}{\mathcal{A}}

\newtheorem{open problem}{Open Problem}

\newcommand{\Loeb}[1]{\overline{#1}}

\newcommand{\interior}[1]{%
  {\kern0pt#1}^{\mathrm{o}}%
}

\newcommand{\refproof}[1]{See \cref{#1} for \IfSubStr{#1}{,}{proofs}{a proof}. }

\newif\iflongform
\longformtrue

\iflongform

\else

\fi

\makeatletter
\providecommand*{\toclevel@definition}{0}
\providecommand*{\toclevel@theorem}{0}
\providecommand*{\toclevel@lemma}{0}
\makeatother
\usepackage{leftidx}

\usepackage{todonotes}

\title[Loeb Extension and Loeb Equivalence II]
{
Loeb Extension and Loeb Equivalence II
}

\subjclass{28E05 (primary), 03H05}

\newcommand{\cM}{\mathcal{M}}
\newcommand{\cN}{\mathcal{N}}
\newcommand{\cH}{\mathcal{H}}

\newcommand{\cV}{\mathcal{V}}

\newcommand{\boundary}[1]{\partial #1}
\newcommand{\symdef}{\triangle}
\newcommand{\cU}{\mathcal{U}}

\newcommand{\mprod}{\otimes}
\newcommand{\smprod}{\otimes^{\sigma}}
\newcommand{\closure}[1]{\bar{#1}}
\newtheorem{question}{Question}
\newcommand{\cov}[2]{\mathrm{cov}_{#1}{(#2)}}
\newcommand{\inter}[2]{\mathrm{int}_{#1}{(#2)}}

\begin{document}

\author[H.~Duanmu]{Haosui Duanmu}
\address{Haosui Duanmu, University of California, Berkeley}

\author[D.~Schrittesser]{David Schrittesser}
\address{David Schrittesser, University of Toronto}

\author[W.~Weiss]{William Weiss}
\address{William Weiss, University of Toronto}

\maketitle

\begin{abstract}

The paper answers two open questions that were raised in \citet{loebsun}.
The first question asks, if we have two Loeb equivalent spaces $(\Omega, \cF, \mu)$ and $(\Omega, \cG, \nu)$, does there exist an internal probability measure $P$ defined on the internal algebra $\cH$ generated from $\cF\cup \cG$ such that $(\Omega, \cH, P)$ is Loeb equivalent to $(\Omega, \cF, \mu)$? 
The second open problem asks if the $\sigma$-product of two $\NSE{\sigma}$-additive probability spaces is Loeb equivalent to the product of the same two $\NSE{\sigma}$-additive probability spaces. 
Continuing the work in  \citep{ADSWloeb}, we give a confirmative answer to the first problem when the underlying internal probability spaces are hyperfinite,  a partial answer to the first problem for general internal probability spaces, and settle the second question negatively by giving a counter-example.  Finally, we show that the continuity sets in the $\sigma$-algebra of the $\sigma$-product space are also in the algebra of the product space.
\end{abstract}

%

Loeb measure spaces, measure spaces constructed using nonstandard analysis according to the method discovered by Loeb \cite{loeb75}, have over the years found applications of in many areas of mathematics: 
For instance, in probability theory (see \citep{anderson87}, \citep{nsweak}),  stochastic processes (see \citep{andersonisrael}, \citep{localtime}, \citep{Keisler87}), Markov processes (see \citep{Markovpaper}, \citep{anderson2018mixhit} etc), statistical decision theory (see \citep{nsbayes}, \citep{nscredible}), and mathematical economics (see \citep{emmons84}, \citep{strongcore}, \citep{secondwelfare}, \citep{indmatching}, \citep{anderson08},  \citep{duffie18}) to name but a few examples.
Although the Loeb measure construction has been well-studied, basic problems about the nature of Loeb measure spaces still remain. 
\citet{loebsun} posed four open problems concerning Loeb extension and Loeb equivalence between internal probability spaces.
In \citep{ADSWloeb}, we provide counter-examples to the first two open problems and a partial solution to the third open problem. 
In this paper, we give an affirmative to the third open problem when the underlying internal probability spaces are hyperfinite. 
We also provide a counter-example for the fourth open problem, hence showing that the completion of the $\sigma$-product space is generally not the same as the completion of the product space. 
However,  we show that continuity sets in the $\sigma$-algebra of $\sigma$-product space are also in the algebra of the product space.  

\medskip

Recall that given an internal probability space $(\Omega, \cF, \mu)$, its Loeb extension is the countably additive probability space $(\Omega, \Loeb{\cF}, \Loeb{\mu})$, where $\Loeb{\cF}$ is the collection of sets $B\subset \Omega$ such that, with $\ST$ denoting the standard part map, 
\[
\sup\{\ST(\mu(A)): B\supset A\in \cF\}=\inf\{\ST(\mu(C)): B\subset C\in \cF\},
\] 
and for every $B\in\Loeb{\cF}$, $\Loeb{\mu}(B)$ is defined to be the supremum (infinum) above. 

\medskip

\citet{loebsun} introduced the following rather straightforward notions of extension and equivalence for pairs of internal probability spaces. 

\begin{definition}\label{defloebext}
Let $\cM=(\Omega, \cF, \mu)$ and $\cN=(\Omega, \cG, \nu)$ be two internal probability spaces. 
We say $\cN$ \textbf{Loeb extends} $\cM$ if $\Loeb{\cG}\supset \Loeb{\cF}$ and $\Loeb{\nu}$ extends $\Loeb{\mu}$ as a function. 
We say $\cN$ is \textbf{Loeb equivalent to} $\cM$ if $\Loeb{\cF}=\Loeb{\cG}$ and $\Loeb{\nu}=\Loeb{\mu}$.  
\end{definition}

If $\cF\subset \cG$ and $\nu$ extends $\mu$ as a function, then it is clear that $\cN$ Loeb extends $\cM$. 
In \citep{loebsun}, Keisler and Sun discuss Loeb extensions and Loeb equivalence in absence of the assumption $\cF\subset \cG$, leaving four open problems to three of which we have provided complete or at least partial solutions in \citep{ADSWloeb}.
Our aim is to give complete answers to the third and fourth open problems, which we state below.

Following \citep{loebsun}, we say an internal probability space 
$(\Omega, \cF, \mu)$ is \textbf{hyperfinite} if $\cF$ is hyperfinite. 
Note that we do not require $\Omega$ to be hyperfinite. 

The following questions asked by \citet{loebsun} remain open: 

\begin{question}\label{loebeqq3}
Suppose $\cM=(\Omega, \cF, \mu)$ is Loeb equivalent to $\cN=(\Omega, \cG, \nu)$, and $\cH$ be the internal algebra generated by $\cF\cup \cG$.
Must there be an internal probability measure $P$ on $\cH$ such that $\cM$ is Loeb equivalent to $(\Omega, \cH, P)$? What if $\cM$ and $\cN$ are assumed to be hyperfinite? 
\end{question}
Let $\cM_1=(M_1, \cF_1, \mu_1)$ and $\cM_2=(M_2, \cF_2, \mu_2)$ be two internal probability spaces. 
Let $\cM_1\mprod \cM_2=(M_1\times M_2, \cF_1\mprod \cF_2, \mu_1\mprod \mu_2)$ be the internal product space, where $\cF_1\mprod \cF_2$ is the internal algebra of all hyperfinite unions of $A_1\times A_2\in \cF_1\times \cF_2$. 
When the internal probability spaces $\cM_1$ and $\cM_2$ are $\NSE{\sigma}$-additive, let $\cM_1\smprod \cM_2$ be the internal $\NSE{\sigma}$-additive probability space by the internal product $\cM_1\mprod \cM_2$. 
\citet{loebsun} ask the following question: 

\begin{question}\label{fthsun}
Let $\cM$ and $\cN$ be internal $\NSE{\sigma}$-additive probability space. Must the product $\cM\mprod \cN$ be Loeb equivalent to the $\sigma$-product $\cM\smprod \cN$?
\end{question}

If $\cM_1$ and $\cM_2$ are standard countably additive probability spaces, let 
\[
\cM_1\mprod \cM_2=(M_1\times M_2, \cF_1\mprod \cF_2, \mu_1\mprod \mu_2)
\]
be the product space, where $\cF_1\mprod \cF_2$ is the algebra generated by $A_1\times A_2\in \cF_1\times \cF_2$. Let 
\[
\cM_1\smprod \cM_2=(M_1\times M_2, \cF_1\smprod \cF_2, \mu_1\smprod \mu_2)
\]
be the $\sigma$-product space. That is, $\cM_1\smprod \cM_2$ is a $\sigma$-additive probability space, where $\cF_1\smprod \cF_2$ is the $\sigma$-algebra generated by all sets of the form $A_1\times A_2\in \cF_1\times \cF_2$. 
In the case that $\cM$ and $\cN$ are nonstandard extensions of standard spaces, \cref{fthsun} takes the following standard form. 

\begin{question}\label{stfthsun}
Let $\cM$ and $\cN$ be standard $\sigma$-additive probability spaces. 
Let $\cM^{c}$ denote the completion of $\cM$. 
Must $(\cM\mprod \cN)^{c}=(\cM\smprod \cN)^{c}$? 
\end{question}

\section{Answer to \cref{loebeqq3}}\label{secondsec}

We give an affirmative answer to \cref{loebeqq3} when both $\cM$ and $\cN$ are hyperfinite.  We start by quoting some results from our earlier paper \citep{ADSWloeb}.

\begin{theorem}[{\citep[][Thm.~3.4]{ADSWloeb}}]\label{unioneq}
Let $(\Omega, \cF, \mu)$ be a hyperfinite probability space and let $\cG$ be a hyperfinite algebra on $\Omega$. 
Let $\cH$ be the internal algebra generated by $\cF\cup \cG$. 
Then $(\Omega, \cH, P)$ is Loeb equivalent to $(\Omega, \cF, \mu)$ for some internal probability measure $P$ if and only if $\cH\subset \Loeb{\cF}$. 
\end{theorem}

By \cref{unioneq}, when $\cM$ and $\cN$ are both hyperfinite, \cref{loebeqq3} is equivalent to the following question: 

\begin{question}\label{posloebquestion}
Suppose $\cM=(\Omega, \cF, \mu)$ and $\cN=(\Omega, \cG, \nu)$ are two hyperfinite probability spaces that are Loeb equivalent to each other. 
Let $\cH$ be the hyperfinite algebra generated by $\cF\cup \cG$. Is $\cH\subset \Loeb{\cF}$? 
\end{question}

As $\cM$ is hyperfinite, there exists an internal subset $\cF_0$ of $\subset$-minimal elements in $\cF\setminus\{\emptyset\}$ which generates $\cF$. 
That is, there exists an internal $\cF_0\subset \cF$ such that
\begin{enumerate}
\item $\cF_0$ is a hyperfinite partition of $\Omega$
\item For every $A\in \cF_0$, if there exists non-empty $E\in \cF$ such that $E\subset A$, then $E=A$.
\end{enumerate}
Clearly any $F\in \cF$ is a hyperfinite union of elements from $\cF_0$ (noting $\emptyset$ is the union of $\emptyset$, which is also a hyperfinite subset of $\cF$). Thus, $\cF_0$ generates the internal algebra $\cF$. 
Similarly, there exists an internal subset $\cG_0$ of $\cG$ such that
\begin{enumerate}
\item $\cG_0$ is a hyperfinite partition of $\Omega$
\item For every $A\in \cG_0$, if there exists non-empty $E\in \cG$ such that $E\subset A$, then $E=A$.
\end{enumerate}
As above, $\cG_0$ generates the internal algebra $\cG$. Let 
\[
\cU=\{A\cap B: A\in \cF_0, B\in \cG_0, A\cap B\neq \emptyset\}. 
\]
Then $\cU$ forms a $\NSE{}$partition of $\Omega$ and every element in $\cH$ can be written as a hyperfinite union of elements in $\cU$. 
As $\cM$ and $\cN$ are Loeb equivalent, it is easy to see that $\cU\subset \Loeb{\cF}$. However, it is not clear whether all hyperfinite (not finite) unions of elements in $\cU$ are elements of $\Loeb{\cF}$. Hence it is not straightforward to determine if $\cH\subset \Loeb{\cF}$. 

\medskip

Let us write
\[
\cF_0'=\{F\in \cF_0: \text{$F$ intersects at least two elements in $\cG_0$}\}.
\] 
The following improves \citep[][Thm.~4.1]{ADSWloeb}.
\begin{lemma}\label{immthm}
Suppose $\cM=(\Omega, \cF, \mu)$ and $\cN=(\Omega, \cG, \nu)$ are two hyperfinite probability spaces that are Loeb equivalent to each other. 
Then $\cH\subset \Loeb{\cF}$ if and only if $\mu(\bigcup \cF_0')\approx 0$. 
\end{lemma}
\begin{proof}
Suppose $\mu(\bigcup \cF_0')\approx 0$. This direction is proven in {\citep[][Thm.~4.1]{ADSWloeb}}. We include the proof for completeness. 
Let $\cF_1=\cF_0\setminus \cF_0'$. 
Pick $H\in \cH$ and, without loss of generality, assume that $H=\bigcup_{i=1}^{K}(A_i\cap B_i)$ where $A_i\in \cF$, $B_i\in \cG$ and $K\in \NSE{\Nats}$.
Let $\cV=\{A_i\cap B_i: i\leq K\}$. 
Then $H=H_1\cup H_2$ where $H_1=\bigcup \{(A_i\cap B_i)\in \cV: A_i\in \cF_0'\}$ and $H_2=\bigcup \{(A_i\cap B_i)\in \cV: A_i\in \cF_1\}$. 
Clearly, $H_1$ is a subset of $\bigcup \cF_0'$. 
As $\mu(\bigcup \cF_0')\approx 0$, by the completeness of Loeb measure, $H_1$ is Loeb measurable and $\Loeb{\mu}(H_1)=0$.
Note that, for every element $F\in \cF_1$, there exists an unique $G\in \cG_0$ such that $F\subset G$. 
Thus, for every $A_i\in \cF_1$, $A_i\cap B_i$ is either $A_i$ or $\emptyset$.
Thus, we know that $H_2\in \cF$.
Hence we conclude that $H\in \Loeb{\cF}$. 

Suppose $\Loeb{\mu}(\bigcup \cF_0')>0$. Let $\cF_0'=\{A_1, A_2, \dotsc, A_m\}$ for some $m\in \NSE{\Nats}$. 
There exists an internal sequence $(B_i)_{i\leq m}$ such that $B_i\in \cG_0$ for all $i\leq m$ and $\emptyset\neq A_i\cap B_i\subsetneq A_i$. 
Let $C=\bigcup_{i=1}^{m}(A_i\cap B_i)$. 
It is easy to see that $C\in \cH$. 
Note that
\[
\bigcap\{A\in \cF: A\supset C\}=\bigcup \cF'_0\  \text{and}\  \bigcup\{A\in \cF: A\subset C\}=\emptyset
\]
Hence $C$ is not $\Loeb{\cF}$-measurable, which implies that $\cH$ is not a subset of $\Loeb{\cF}$. 
\end{proof}

We now prove the main result of this section. 

\begin{theorem}\label{covintresult}
Suppose $\cM=(\Omega, \cF, \mu)$ and $\cN=(\Omega, \cG, \nu)$ are a pair of Loeb equivalent hyperfinite probability spaces.  
Then $\mu(\bigcup \cF_0')\approx 0$. 
\end{theorem}
\begin{proof}
For every internal set $T\subset \Omega$, we define $\cov{\cF_0}{T}=\bigcup\{F\in \cF_0: F\cap T\neq \emptyset\}$ and $\inter{\cF_0}{T}=\bigcup\{F\in \cF_0: F\subset T\}$. 
We define $\cov{\cG_0}{T}$ and $\inter{\cG_0}{T}$ similarly. 
We now construct a non-increasing sequence $\{F_n': n\geq 0\}\subset \cF$ recursively. 
Define $F_0'=\bigcup \cF_0'$. 
Suppose that we have constructed $F_n'$, we now construct $F_{n+1}'$. 

\textbf{Case 1}: Suppose for every $A\in \cF_0\cap \PowerSet (F_n')$, there exists $B\in \cG_0$ such that $B\subset A$. Note that $F_n'\subset F_0'=\bigcup \cF_0'$. 
As $\cF_0'$ is a hyperfinite set, without loss of generality, assume $F_n'=\bigcup_{i\leq K}A_{i}$, where $K\in \NSE{\Nats}$ and $A_i\in \cF_0'$ for all $i\leq K$.  
For each $i\leq K$, pick $B_i\in \cG_0$ such that $B_i\subset A_i$. 
As $A_i\in \cF_0'$,  we know that $B_i\neq A_i$, which implies that $\bigcup_{i\leq K}B_{i}\subsetneq \bigcup_{i\leq K}A_i=F_n'$. 
Note that $\inter{\cF_0}{\bigcup_{i\leq K}B_i}=\emptyset$ and $\cov{\cF_0}{\bigcup_{i\leq K}B_i}=F_n'$. 
Thus, we conclude that $\mu(F_n')\approx 0$. 
In this case, for every $m>n$, we simply let $F'_m=\emptyset$.

\textbf{Case 2}: Suppose that there exists $A_n^{*}\in \cF_0\cap \PowerSet (F_n')$ such that no $G\in \cG_0$ is a subset of $A_n^{*}$. 
Clearly, $\cov{\cG_0}{F_n'}\setminus \cov{\cG_0}{A_n^{*}}$ is a proper subset of $\cov{\cG_0}{F_n'}$. 
Define $F_{n+1}'=\inter{\cF_0}{F_n'\setminus \cov{\cG_0}{A_n^{*}}}$. 
Note that $F_{n+1}'$ is a proper subset of $F_{n}'$, and $F_{n}'\setminus F_{n+1}'$ contains an element of 
$\cF_0$ (namely, $A_n^*$) as a subset. 
\begin{claim}\label{f0in}
$F_n'\subset F_{n+1}'\cup \cov{\cF_0}{\cov{\cG_0}{A_n^{*}}}$. 
\end{claim}
\begin{proof}
Let $A$ be an element of $\cF_0\cap \PowerSet(F_n')$. 
Suppose that $A$ is not a subset of $\cov{\cF_0}{\cov{\cG_0}{A_n^{*}}}$. 
This implies that $A\cap \cov{\cG_0}{A_n^{*}}=\emptyset$.
As $A\in \PowerSet(F_n')$, we conclude that $A\subset F_n'\setminus \cov{\cG_0}{A_n^{*}}$, hence implies that $A\subset F_{n+1}'$.
\end{proof}

We continue this process to construct an internal sequence $\{F_n': n\geq 0\}$.
As $F'_{n+1}\subsetneq F'_n$ and $F'_n\setminus F'_{n+1}$ contains some element in $\cF_0$ as a subset, 
there exists some $M\in \NSE{\Nats}$ such that Case 1 holds for $F'_M$.
In particular, for every $A\in \cF_0\cap \PowerSet (F_M')$, there exists $B\in \cG_0$ such that $B\subset A$. 
By the argument in Case 1, we conclude that $\mu(F_M')\approx 0$. 
Note that we also have $\{A_i^{*}: i<M\}$ as defined in Case 2. 
\begin{claim}\label{unionzero}
$\bigcup_{i=0}^{M-1}\cov{\cF_0}{\cov{\cG_0}{A_i^{*}}}$ is a $\Loeb{\mu}$-null set.  
\end{claim}
\begin{proof}
For $i,j<M-1$ such that $i\neq j$, without loss of generality, assume that $i<j$. 
Then, we have $A_{j}^{*}\subset \inter{\cF_0}{F'_{i}\setminus \cov{\cG_0}{A_i^{*}}}$. 
Suppose $\cov{\cG_0}{A_i^{*}}\cap \cov{\cG_0}{A_j^{*}}\neq \emptyset$.
Then there exists some $G\in \cG_0$ such that $G$ is contained in both $\cov{\cG_0}{A_i^{*}}$ and $\cov{\cG_0}{A_j^{*}}$ as a subset. This implies that $A_{j}^{*}\cap \cov{\cG_0}{A_i^{*}}\neq \emptyset$, hence a contradiction. 
So we have $\cov{\cG_0}{A_i^{*}}\cap \cov{\cG_0}{A_j^{*}}=\emptyset$. 
Note that, for every $i<M-1$, $A_i^{*}$ contains no $G\in \cG_0$ as a subset. 
Hence, we have $\inter{\cG_0}{\bigcup_{i=0}^{M-1}A_i^{*}}=\emptyset$. 
This implies that $\Loeb{\mu}(\bigcup_{i=0}^{M-1}A_i^{*})=\Loeb{\nu}(\bigcup_{i=0}^{M-1}A_i^{*})=0$. 
As $\bigcup_{i=0}^{M-1}\cov{\cG_0}{A_i^{*}}\subset \cov{\cG_0}{\bigcup_{i=0}^{M-1}A_i^{*}}$, 
we have 
\[
\Loeb{\nu}(\bigcup_{i=0}^{M-1}\cov{\cG_0}{A_i^{*}})=0, 
\]
which implies that $\cov{\cF_0}{\bigcup_{i=0}^{M-1}\cov{\cG_0}{A_i^{*}}}$ is a $\Loeb{\mu}$-null set. 
So $\bigcup_{i=0}^{M-1}\cov{\cF_0}{\cov{\cG_0}{A_i^{*}}}$ is $\Loeb{\mu}$-null. 
\end{proof}

Note that $F_n'\subset F_{n+1}'\cup \cov{\cF_0}{\cov{\cG_0}{A_i^{*}}}$. 
Thus, $F_0'\subset \bigcup_{n=0}^{M-1}\cov{\cF_0}{\cov{\cG_0}{A_i^{*}}}\cup F_{M}'$. 
Hence, we conclude that $\mu(\bigcup \cF_0')=\mu(F_0')\approx 0$, completing the proof. 
\end{proof}

Combining \cref{immthm}, \cref{covintresult} and \cref{unioneq}, we give an affirmative answer to \cref{loebeqq3} when both $\cM$ and $\cN$ are hyperfinite. 
\begin{theorem}\label{loebeqq3main}
Suppose $\cM=(\Omega, \cF, \mu)$ and $\cN=(\Omega, \cG, \nu)$ are two hyperfinite probability spaces that are Loeb equivalent to each other. 
Let $\cH$ be the hyperfinite algebra generated by $\cF\cup \cG$.
Then there is an internal probability measure $P$ on $\cH$ such that $\cM$ is Loeb equivalent to $(\Omega, \cH, P)$. 
\end{theorem}

\section{\cref{loebeqq3} for General Internal Probability Spaces}

In this section, we give a partial answer to \cref{loebeqq3} when $\cM$ and $\cN$ are general internal probability spaces. 
We start with the following definition:
\begin{definition}
Let $X$ be an internal space equipped with an internal algebra $\cA$. 
An internal almost probability measure $P$ is an internal mapping from $\cA$ to $\NSE{[0,1]}$ such that:
\begin{enumerate}
\item $P(\emptyset)=0$ and $P(X)=1$;
\item For every $A, B\in \cA$, $P(A\cup B)\approx P(A)+P(B)$.
\end{enumerate}
\end{definition}

Note that internal almost probability spaces may not be internal probability spaces. One natural question to ask is:
\begin{question}\label{intamstintclose}
Is every internal almost probability space ``close" to some internal probability space ? 
\end{question}

As is well-known, Loeb probability measures are constructed from internal probability measures. 
The next lemma shows that we can construct Loeb probability measures from internal almost probability measures. 
The proof is almost identical to the proof of the existence of the Loeb measure (see {\citep[][Section.~4, Theorem.~2.1]{NSAA97}}). 
We give a full proof for completeness. 

\begin{lemma}\label{amstextend}
Let $(X, \cA, P)$ be an internal almost probability space. 
Then there exists a standard countably additive probability space $(X, \Loeb{\cA}, \Loeb{P})$ such that:
\begin{enumerate}
\item $\Loeb{\cA}$ is a $\sigma$-algebra with $\cA\subset \Loeb{\cA}\subset \PowerSet(X)$;
\item $\Loeb{P}=\ST(P)$ on $\cA$;
\item For every $A\in \Loeb{\cA}$ and every standard $\epsilon>0$, there exist $A_i, A_o\in \cA$ such that $A_i\subset A\subset A_o$ and $P(A_o\setminus A_i)<\epsilon$;
\item For every $A\in \Loeb{\cA}$, there is a $B\in \cA$ such that $\Loeb{P}(A\symdef B)=0$

\end{enumerate}

\end{lemma}

\begin{proof}
Note that $(X, \cA, \ST(P))$ is a finitely additive probability space. 
Let $A_0\supset A_1\supset A_2\dotsc \supset A_n\dotsc$ be a countable non-increasing sequence of elements in $\cA$ such that $\bigcap_{n\in \Nats}A_n=\emptyset$.  By saturation,  $A_N=\emptyset$ for all $N\in \NSE{\Nats}\setminus \Nats$. Thus, we have $\lim_{n\to \infty}\ST(P)(A_n)=\ST(P)(A_N)=0$. So the first three items follow from the Caratheodory extension theorem. 

We now show that the fourth item is also valid. 
Pick $A\in \Loeb{\cA}$. 
For every $n\in \Nats$, there exist $A_i^{n}, A_o^{n}\in \cA$ such that $A_i^{n}\subset A\subset A_o^{n}$ and 
$P(A_o^{n}\setminus A_i^{n})<\epsilon$. 
Without loss of generality, we can assume that the sequence $\{A_i^{n}\}_{n\in \Nats}$ is a non-decreasing sequence, and the sequence $\{A_o^{n}\}_{n\in \Nats}$ is a non-increasing sequence.  
By saturation, there exists $B\in \cA$ such that $A_i^{n}\subset B\subset A_o^{n}$ for all $n\in \Nats$. 
Then, for each $n\in \Nats$, we have
\[
\Loeb{P}(A\symdef B)\leq 2\Loeb{P}(A_o^{n}\setminus A_i^{n})\approx \frac{2}{n}.
\]
Hence, we have $\Loeb{P}(A\symdef B)=0$. 
\end{proof}

As one can apply the Loeb construction to internal almost probability measures, for an internal almost probability space $(X, \cA, P)$, with a slight abuse of notation, we call the standard probability space $(\Omega, \Loeb{\cA}, \Loeb{P})$ the Loeb space generated from $(X, \cA, P)$.  
The notions of Loeb extension and equivalence extend naturally to internal almost probability spaces. 

\begin{definition}
Let $\cM=(\Omega, \cF, \mu)$ and $\cN=(\Omega, \cG, \nu)$ be two internal almost probability spaces. 
$\cN$ Loeb extends $\cM$ if $\Loeb{\cF}\subset \Loeb{\cG}$ and $\Loeb{\nu}$ extends $\Loeb{\mu}$ as a function. 
$\cN$ is Loeb equivalent to $\cM$ if $\Loeb{\cF}=\Loeb{\cG}$ and $\Loeb{\nu}=\Loeb{\mu}$. 
\end{definition}

\begin{theorem}\label{amstprobexst}
Suppose $\cM=(\Omega, \cF, \mu)$ be an internal almost probability space and let $\cG$ be an internal algebra on $\Omega$.
Then $(\Omega, \cF, \mu)$ Loeb extends $(\Omega, \cG, P)$ for some internal almost probability measure $P$ if and only if $\cG\subset \Loeb{\cF}$. 
\end{theorem}
\begin{proof}
Clearly, if there exists such an internal almost probability measure $P$, then $\cG\subset \Loeb{\cF}$. 
Now, suppose that $\cG\subset \Loeb{\cF}$. 
For every $\epsilon>0$ and every $G\in \cG$, there exist $G_i^{\epsilon}, G_o^{\epsilon}\in \cF$ such that $G_i^{\epsilon}\subset G\subset G_o^{\epsilon}$ and $\mu(G_o^{\epsilon}\setminus G_i^{\epsilon})<\epsilon$. 
By underspill, there exists some positive infinitesimal $\epsilon_0$ such that, for every $G\in \cG$, there exist $G_i^{\epsilon_0}, G_o^{\epsilon_0}\in \cF$ such that $G_i^{\epsilon_0}\subset G\subset G_o^{\epsilon_0}$ and $\mu(G_o^{\epsilon_0}\setminus G_i^{\epsilon_0})<\epsilon_0$. Define an internal mapping $P_0: \cG\to \NSE{[0,1]}$ by letting $P_0(G)=\mu(G)$ if $G\in \cF$ and $P_0(G)=\mu(G_i^{\epsilon_0})$ otherwise. 
\begin{claim}
$P_0$ is an internal almost probability measure on $(\Omega, \cG)$. 
\end{claim} 
\begin{proof}
Clearly, we have $P_0(\emptyset)=0$ and $P_0(\Omega)=1$. 
For $A, B\in \cG$, we have
\[
P_0(A\cup B)\approx \mu(A\cup B)\approx \mu(A)+\mu(B)\approx P_0(A)+P_0(B).
\]
\end{proof}
We now show that $(\Omega, \cF, \mu)$ Loeb extends $(\Omega, \cG, P_0)$. 
It is sufficient to show that $\Loeb{P_0}$ and $\Loeb{\mu}$ agree on $\cG$. 
For every $G\in \cG$, we have 
\[
\Loeb{P_0}(G)\approx P_0(G)\approx \mu(G_i^{\epsilon_0})\approx \Loeb{\mu}(G). 
\]
Hence, we have the desired result. 
\end{proof}

The following theorem provides a partial answer to \cref{loebeqq3} for general internal probability spaces. 

\begin{theorem}\label{amsinternalapprox}
Let $(\Omega, \cF, \mu)$ be an internal probability space and let $\cG$ be an internal algebra on $\Omega$. 
Let $\cH$ be an internal algebra generated by $\cF\cup \cG$. 
Then $(\Omega, \cH, P)$ is Loeb equivalent to $(\Omega, \cF, \mu)$ for some internal almost probability measure $P$ if and only if $\cH\subset \Loeb{\cF}$. 
\end{theorem}
\begin{proof}
Suppose there exists such an internal almost probability measure $P$, then we clearly have $\cH\subset \Loeb{\cF}$, which further implies that $\Loeb{\cH}=\Loeb{\cF}$. 
Now suppose that $\cH\subset \Loeb{\cF}$. 
By \cref{amstprobexst}, there exists an internal almost probability measure $P$ on $\cH$ such that $(\Omega, \cF, \mu)$ Loeb extends $(\Omega, \cH, P)$. 
As $\cF\subset \cH$, we have $\Loeb{P}(A)=\Loeb{\mu}(A)$ for all $A\in \cF$. 
Hence, we have the desired result. 
\end{proof} 

A natural question to ask is: 
\begin{question}\label{internalapprox}
Let $(\Omega, \cF, \mu)$ be an internal probability space and let $\cG$ be an internal algebra on $\Omega$. 
Let $\cH$ be an internal algebra generated by $\cF\cup \cG$. 
Suppose that $\cH\subset \Loeb{\cF}$, does there exist an internal probability measure $P$ such that $(\Omega, \cH, P)$ is Loeb equivalent to $(\Omega, \cF, \mu)$?
\end{question}

A confirmative answer of \cref{intamstintclose} is likely to give a confirmative answer to \cref{internalapprox}. 
\cref{amsinternalapprox} and \cref{internalapprox} are related to the following more general question: 
\begin{question}\label{reverseloeb}
Let $\Omega$ be an internal space and $\cF$ be an internal algebra. 
Let $\sigma(\cF)$ be the $\sigma$-algebra generated from $\cF$ and $P$ be a probability measure on $\sigma(\cF)$. 
Does there exist an (almost) internal probability measure $\mu$ on $\cF$ such that the Loeb extension $\Loeb{\mu}$ agrees with $P$ on $\sigma(\cF)$?
\end{question}

The Loeb construction allows one to construct a standard countably additive probability measure from an internal probability measure. 
\cref{reverseloeb} asks if every standard probability measure on the $\sigma$-algebra generated from the internal algebra is the Loeb extension of some internal probability measure.

\section{A Comprehensive Answer to \cref{fthsun}}

In this section, we provide a negative answer to \cref{stfthsun} and hence also to \cref{fthsun}. 
On the other hand, we show that every continuity set in the $\sigma$-algebra of the $\sigma$-product space is also in the completion of the finite product space. We start by introducing the following definition:  

\begin{definition}\label{hyperapproxsp}
Let $(X,d)$ be a compact metric space with Borel $\sigma$-algebra $\BorelSets X$.
A hyperfinite representation of $X$ is a tuple $(S_X,\{B_X(s)\}_{s\in S_X})$ such that

\begin{enumerate}
\item $S_X$ is a hyperfinite subset of $\NSE{X}$.
\item $s\in B_X(s)\in \NSE{\BorelSets X}$ for every $s\in S_X$.
\item For every $s\in S_X$, the diameter of $B_X(s)$ is infinitesimal.
\item The hyperfinite collection $\{B_X(s: x\in S_X)\}$ forms a *partition of $\NSE{X}$. 
\end{enumerate}
For every $x\in \NSE{X}$, we use $s_{x}$ to denote the unique element in $S$ such that $x\in B(s_x)$.
\end{definition}

The next theorem shows that hyperfinite representations exist for product spaces under moderate conditions. 
Its proof is almost identical to {\citep[][Thm.~6.6]{Markovpaper}}, hence is omitted. 

\begin{theorem}\label{exhyper}
Let $X, Y$ be two compact metric space with Borel $\sigma$-algebras $\BorelSets X$ and $\BorelSets Y$, respectively. 
Then there exists a hyperfinite representation $(S, \{B(s)\}_{s\in S})$ of $X\times Y$ 
such that $B(s)\in \NSE{\BorelSets X}\mprod \NSE{\BorelSets Y}$ for every $s\in S$. 
\end{theorem}

Loeb measures can usually be used to represent standard measures via the standard part map. 

\begin{theorem}[{\citep[][Thm.~3.3]{anderson87}}]\label{bobrepresent}
Let $(X, \cF, \mu)$ be a Radon probability space. Then, for every $E\in \cF$, we have $\mu(E)=\Loeb{\NSE{\mu}}(\ST^{-1}(E))$. 
\end{theorem}


%
%

The completion of a probability space $(X, \cF, \mu)$ can be defined in the following to equivalent ways:
\begin{enumerate}
\item The completion $\cF^{c}$ is the collection of all sets of the form $A\cup C$, where $A\in \cF$ and $C$ is a subset of some 
$B\in \cF$ with $\mu(B)=0$;

\item The completion $\cF^{c}$ is the collection of all sets $A\subset X$ such that, for all $\epsilon>0$, there exists 
$A_1, A_2\in \cF$ with $A_1\subset A\subset A_2$ and $\mu(A_2\setminus A_1)<\epsilon$. 
\end{enumerate}

If $(X, \cF, \mu)$ is a finitely additive probability space, then these two definitions of completion are not the same. 
To see this, we consider the following example: 

\begin{example}
Let $\cM_0$ be $([0,1], \BorelSets {[0,1]})$ with uniform probability measure $\mu$. 
Let $D=\{(x,x): x\in [0,1]\}$ be the diagonal of $[0,1]\times [0,1]$. 
Clearly, we have $D\in (\cM_0\smprod \cM_0)^{c}$ and $(\mu\smprod \mu)(D)=0$. 
We now show that $D$ is not contained in any null set in $\cM_0\mprod \cM_0$.  

Suppose it is. Then $D$ must be contained in some $E\in \cM_0\mprod \cM_0$ such that $(\mu\mprod \mu)(E)=0$. 
Without loss of generality, assume that $E=\bigcup_{i=1}^{n}(A_i\times B_i)$ where $i\in \Nats$ and $A_i, B_i\in \BorelSets {[0,1]}$. 
For $i\leq n$, let $D_i=D\cap (A_i\times B_i)=(A_i\cap B_i)\times (A_i\cap B_i)$.
It is clear that $D=\bigcup_{i=1}^{n}D_i$ hence $\bigcup_{i=1}^{n} (A_i\cap B_i)=[0,1]$. 
As $(\mu\mprod \mu)(E)=0$, we know that $(\mu \mprod \mu)(A_i\times B_i)=0$ for every $i\leq n$. 
This immediately implies that $\mu(A_i\cap B_i)=0$ for every $i\leq n$. 
However, as $\bigcup_{i=1}^{n} A_i\cap B_i=[0,1]$, we must have $\mu(\bigcup_{i=1}^{n} (A_i\cap B_i))=1$ which clearly leads to a contradiction. Thus, $D$ is not contained in any null set in $\cM_0\mprod \cM_0$.

On the other hand, for every $\epsilon>0$, there exists $F\in \cM_0\mprod \cM_0$ such that $D\subset F$ and $\mu(F)<\epsilon$. 
\end{example}

For the rest of this section, when $(X, \cF, \mu)$ is a finitely additive probability space, the completion $\cF^{c}$ is taken to be 
the collection of all sets $A\subset X$ such that, for all $\epsilon>0$, there exists 
$A_1, A_2\in \cF$ with $A_1\subset A\subset A_2$ and $\mu(A_2\setminus A_1)<\epsilon$. 
We now prove the first main result of this section. 

\begin{theorem}\label{falgeeq}
Let $\cA_1=(X, \BorelSets X, P_1)$ and $\cA_2=(Y, \BorelSets Y, P_2)$ 
be two Borel probability spaces such that both $X$ and $Y$ are compact metric spaces. 
Then every $P_1\smprod P_2$-continuity set is an element of $(\BorelSets X\mprod \BorelSets Y)^{c}$.

\end{theorem}
\begin{proof} 
Suppose not. 
Then there exists a $(P_1\smprod P_2)$-continuity set $E\in \BorelSets X\smprod \BorelSets Y$ and $\epsilon>0$ such that 
\[
\neg(\exists A, B\in \BorelSets X\mprod \BorelSets Y)(A\subset E\subset B\wedge (P_1\mprod P_2)(B\setminus A)<\epsilon). 
\]
By the transfer principle, we have 
\[
\neg(\exists A, B\in \NSE{\BorelSets X}\mprod \NSE{\BorelSets Y})(A\subset \NSE{E}\subset B\wedge (\NSE{P_1}\mprod \NSE{P_2})(B\setminus A)<\epsilon). 
\]

By \cref{exhyper}, we can fix a hyperfinite representation $(S, \{B(s)\}_{s\in S})$ of $X\times Y$ 
such that $B(s)\in \NSE{\BorelSets X}\mprod \NSE{\BorelSets Y}$ for every $s\in S$. 
Let $d_1$ and $d_2$ denote the metric of $X$ and $Y$, respectively. 
Let $d$ be a metric on $X\times Y$ such that $d((x_1, x_2), (y_1, y_2))=\max\{d_1(x_1, y_1), d_2(x_2, y_2)\}$. 
For $x\in X\times Y$, let $U_{x}(n)=\{y\in X\times Y: d(x,y)<\frac{1}{n}\}$. 
Clearly, $U_{x}(n)\in \BorelSets X \mprod \BorelSets Y$. 
Let $\closure{E}$ denote the closure of $E$. 
For every $n\in \Nats$, as $\closure{E}$ is compact, there exists a finite subcover from $\{U_{x}(n): x\in E\}$ of $\closure{E}$. 
Let $E_n$ be the union of such a finite subcover. 
Then, for every $n$, $E_n$ is an element of $\BorelSets X \mprod \BorelSets Y$. 
Moreover, as $\bigcap_{n\in \Nats}E_n=\closure{E}$, there exists $n_0\in \Nats$ such that $(P_1\smprod P_2)(E_{n_0}\setminus \closure{E})<\frac{\epsilon}{100}$. As $E$ is a $(P_1\smprod P_2)$-continuity set, we know that $(P_1\smprod P_2)(E_{n_0}\setminus E)<\frac{\epsilon}{100}$. By the transfer principle, $\NSE{E}\subset \NSE{E_{n_0}}\in \NSE{\BorelSets X} \mprod \NSE{\BorelSets Y}$ and $(\NSE{P_1}\smprod \NSE{P_2})(\NSE{E_{n_0}}\setminus \NSE{E})<\frac{\epsilon}{100}$.

As $\cA_1\smprod \cA_2$ is Radon, there exists a compact set $C\subset E$ such that 
\[
(P_1\smprod P_2)(E\setminus C)<\frac{\epsilon}{100}. 
\]
Let $S_{C}=\{s\in S: (\exists x\in \NSE{C})(x\in B(s))\wedge (B(s)\subset \NSE{E})\}$. 
Clearly, $S_{C}$ is a subset of $\NSE{E}$.
By the internal definition principle, $S_{C}$ is hyperfinite. 
Thus, by \cref{exhyper}, $\bigcup_{s\in S_{C}}B(s)$ is an element of $\NSE{\BorelSets X}\mprod \NSE{\BorelSets Y}$. 

\begin{claim}\label{compactclaim}
\[
(\NSE{P_1}\mprod \NSE{P_2})(\bigcup_{s\in S_{C}}B(s))\approx (P_1\smprod P_2)(C).
\]
\end{claim}
\begin{proof}
Let $B_{C}=\{s\in S: (\exists x\in \NSE{C})(x\in B(s))\}$. 
Clearly, $\NSE{C}$ is a subset of $\bigcup_{s\in B_{C}}B(s)$.
As $C$ is compact, every element in $\NSE{C}$ is infinitely close to some element in $C$. 
Clearly, every element in $\bigcup_{s\in B_{C}}B(s)$ is infinitely close to some element in $C$. 
Thus, we have $\bigcup_{s\in B_{C}}B(s)\subset \ST^{-1}(C)$. 

By \cref{bobrepresent}, we have $(P_1\smprod P_2)(C)=\Loeb{(\NSE{P_1}\smprod \NSE{P_2})}(\ST^{-1}(C))$. 
Thus, we can conclude that $(P_1\smprod P_2)(C)\approx (\NSE{P_1}\mprod \NSE{P_2})(\bigcup_{s\in B_{C}}B(s))$. 
For every element $s$ in $B_{C}\setminus S_{C}$, as both $X$ and $Y$ are compact, $s$ is infinitely close to some element in $\boundary{E}$. 
Hence, we know that $\bigcup_{s\in B_{C}\setminus S_{C}}B(s)\subset \ST^{-1}(\boundary{E})$. 
As $E$ is a $(P_1\smprod P_2)$-continuity set, we know that $(P_1\smprod P_2)(\boundary{E})=0$. 
By \cref{bobrepresent}, we conclude that 
\[
(\NSE{P_1}\mprod \NSE{P_2})(\bigcup_{s\in B_{C}\setminus S_{C}}B(s))\lessapprox \Loeb{(\NSE{P_1}\smprod \NSE{P_2})}(\ST^{-1}(\boundary{E}))=0. 
\]
Hence, we conclude that $(\NSE{P_1}\mprod \NSE{P_2})(\bigcup_{s\in S_{C}}B(s))\approx (P_1\smprod P_2)(C)$.
\end{proof}


Note that both $\bigcup_{s\in S_{C}}B(s)$ and $\NSE{E_{n_0}}$ 
are elements of $\NSE{\BorelSets X}\mprod \NSE{\BorelSets Y}$.
Moreover, we have $\bigcup_{s\in S_{C}}B(s)\subset \NSE{E}\subset \NSE{E_{n_0}}$ such that 
\[
&(\NSE{P_1}\mprod \NSE{P_2})(\NSE{E_{n_0}}\setminus \bigcup_{s\in S_{C}}B(s))\\
&=(\NSE{P_1}\smprod \NSE{P_2})(\NSE{E_{n_0}}\setminus \bigcup_{s\in S_{C}}B(s))\\
&=(\NSE{P_1}\smprod \NSE{P_2})(\NSE{E_{n_0}}\setminus \NSE{E})+(\NSE{P_1}\smprod \NSE{P_2})(\NSE{E}\setminus \bigcup_{s\in S_{C}}B(s))<\epsilon
\]
This is a contradiction, hence $E$ is an element of $(\BorelSets X\mprod \BorelSets Y)^{c}$, completing the proof. 
\end{proof} 

\cref{falgeeq} also provides a partial answer to \cref{fthsun}. 

\begin{theorem}\label{nsfour}
Let $\cA_1=(X, \BorelSets X, P_1)$ and $\cA_2=(Y, \BorelSets Y, P_2)$ 
be two Borel probability spaces such that both $X$ and $Y$ are compact metric spaces. 
Then every $\NSE{P_1}\smprod \NSE{P_2}$-continuity set is an element of $\Loeb{(\NSE{\BorelSets X}\mprod \NSE{\BorelSets Y})}$.
\end{theorem}
\begin{proof} 
Pick a standard $\epsilon>0$ and let $E$ be a $\NSE{P_1}\smprod \NSE{P_2}$-continuity set. 
By the transfer of \cref{falgeeq}, there exist $A, B\in (\NSE{\BorelSets X}\mprod \NSE{\BorelSets Y})$ such that $A\subset E\subset B$ and $(\NSE{P_1}\mprod \NSE{P_2})(B\setminus A)<\epsilon$. 
Thus, $E$ is an element of $\Loeb{(\NSE{\BorelSets X}\mprod \NSE{\BorelSets Y})}$.
\end{proof}

However, \cref{stfthsun} is generally false. 
We first introduce the following definition:

\begin{definition}\label{defiso}
For any given Borel measure spaces $(X, \BorelSets X, P_1)$ and $(Y, \BorelSets Y, P_2)$, 
a measurable, measure preserving bijection $F: X\to Y$ is called an isomorphism if
\begin{enumerate}
\item The set of discontinuity points of $F$ is contained in some $B_1\in \BorelSets X$ with $P_1(B_1)=0$;
\item The set of discontinuity points of $F^{-1}$ is contained in some $B_2\in \BorelSets Y$ with $P_2(B_2)=0$.
\end{enumerate}
\end{definition}

By Theorem 1 of \citep{sun1996}, every Polish space endowed with the Borel $\sigma$-algebra and an atomless probability measure is Borel isomorphic to the unit interval with the Lebesgue measure.
We now prove the following theorem:

\begin{theorem}[Communicated by David Fremlin]\label{stfthfalse}
Let $X$ and $Y$ be two Polish spaces endowed with Borel $\sigma$-algebras $\BorelSets X$ and $\BorelSets Y$, two atomless probability measures $\mu$ and $\nu$, respectively.  
Then there exists a set $E\in \BorelSets X \smprod \BorelSets Y$ such that $E$ is not an element of $(\BorelSets X\mprod \BorelSets Y)^{c}$.
\end{theorem}
\begin{proof}
By Theorem 1 of \citep{sun1996}, $(X, \BorelSets X, \mu)$ and $(Y, \BorelSets Y, \nu)$ are Borel isomorphic to the unit interval with the Lebesgue measure. Then there exist 
\begin{enumerate}
\item An independent sequence $\{E_n: n\in \Nats\}\subset \BorelSets X$ such that $\mu(E_n)=\frac{1}{n+1}$ for all $n\in \Nats$;
\item An independent sequence $\{F_n: n\in \Nats\}\subset \BorelSets Y$ such that $\mu(F_n)=\frac{1}{n+1}$ for all $n\in \Nats$.
\end{enumerate}
Let $E=(X\times Y)\setminus (\bigcup_{n\in \Nats}E_n\times F_n)=\bigcap_{n\in \Nats}\big((X\times Y)\setminus (E_n\times F_n)\big)$. Thus, $E$ is an element of $\BorelSets X\smprod \BorelSets Y$ and $(\mu \smprod \nu)(E)\geq \prod_{n\in \Nats}(1-\frac{1}{n+1})^{2}>0$. To finish the proof, it is sufficient to show that $E$ does not contain any set of the form $A\times B$ with $A\in \BorelSets X$, $B\in \BorelSets Y$ and $(\mu \mprod \nu)(A\times B)>0$. 

Let $A\in \BorelSets X$ and $B\in \BorelSets Y$ be such that $\mu(A)>0$ and $\nu(B)>0$. 
Let $I=\{n\in \Nats: \mu(A\cap E_n)=0\}$ and $J=\{n\in \Nats: \nu(B\cap F_n)=0\}$. 
For any $i\in I$, we have $\mu(X\setminus E_{i})>\mu(A)>0$. 
Hence, we have $\prod_{n\in I}\mu(X\setminus E_{n})=\prod_{n\in I}(1-\frac{1}{n+1})>0$. 
Similarly, we have $\prod_{n\in J}\nu(X\setminus F_{n})=\prod_{n\in J}(1-\frac{1}{n+1})>0$. 
So, there exists some $m\in \Nats$ such that $m\not\in I\cup J$, which futher implies that 
$(\mu \mprod \nu)\big((A\times B)\cap (E_{m}\times F_{m})\big)>0$. 
This means that $A\times B$ can not be contained in $E$ so $E$ is not an element of $(\BorelSets X\mprod \BorelSets Y)^{c}$.
\end{proof}

By \cref{stfthfalse}, we conclude that the answers to \cref{fthsun} and \cref{stfthsun} are generally negative.

\printbibliography

\end{document}